\documentclass[12pt, reqno]{amsart}
\usepackage{amsmath, amsthm, amscd, amsfonts, amssymb, mathrsfs, graphicx, color}
\usepackage[bookmarksnumbered, colorlinks, plainpages]{hyperref}

\newtheorem{theorem}{Theorem}[section]
\newtheorem{lemma}[theorem]{Lemma}
\newtheorem{proposition}[theorem]{Proposition}

\theoremstyle{definition}
\newtheorem{definition}[theorem]{Definition}
\newtheorem{example}[theorem]{Example}

\theoremstyle{remark}
\newtheorem{remark}[theorem]{Remark}
\numberwithin{equation}{section}

\begin{document}
\setcounter{page}{1}

\title[Operator convexity in Krein spaces]{Operator convexity in Krein spaces}

\author[M.S. Moslehian and M. Dehghani]{M. S. Moslehian$^1$ and M. Dehghani $^{1,2}$}

\address{$^{1}$ Department of Pure Mathematics, Center of Excellence in
Analysis on Algebraic Structures (CEAAS), Ferdowsi University of
Mashhad, P. O. Box 1159, Mashhad 91775, Iran}
\email{\textcolor[rgb]{0.00,0.00,0.84}{moslehian@um.ac.ir,
moslehian@member.ams.org}}

\address{$^{2}$ Department of Pure and Applied Mathematics, University of Yazd, Yazd 89195-741, Iran}
\email{\textcolor[rgb]{0.00,0.00,0.84}{e.g.mahdi@gmail.com}}



\begin{abstract}
We introduce the notion of Krein-operator convexity in the setting
of Krein spaces. We present an indefinite version of the Jensen
operator inequality on Krein spaces by showing that if
$(\mathscr{H},J)$ is a Krein space, $\mathcal{U}$ is an open set
which is symmetric with respect to the real axis such that
$\mathcal{U}\cap\mathbb{R}$ consists of a segment of real axis and
$f$ is a Krein-operator convex function on $\mathcal{U}$ with
$f(0)=0$, then
\begin{eqnarray*}
f(C^{\sharp}AC)\leq^{J}C^{\sharp}f(A)C
\end{eqnarray*}
for all $J$-positive operators $A$ and all invertible $J$-contractions $C$
such that the spectra of
$A$, $C^{\sharp}AC$ and $D^{\sharp}AD$ are contained in $\mathcal{U}$,
where $D$ is a defect operator for $C^{\sharp}$.\\
We also show that in contrast with usual operator convex functions
the converse of this implication is not true, in general.
\end{abstract} \maketitle

\section{Introduction and preliminaries}
\noindent Linear spaces with indefinite inner products were used
for the first time in the quantum field theory in physics by Dirac
\cite{Dirac} and Pauli \cite{Pa}. Their first mathematical
definition was provided by Pontrjagin \cite{Po} and since then
they have been studied by many mathematicians. Krein spaces as an
indefinite generalization of Hilbert spaces were formally defined
by Ginzburg \cite{Gi} and applies in the quantum field theory by
Jak\'obczyk \cite{JAK} and others. Strohmaier \cite{STR} applied
Krein spaces in the definition of semi-Riemannian spectral triples
in noncommutative geometry. Bebiano et al. \cite{Bebiano2} proved
an extension of the classical theory of Courant and Fisher to the
case of $J$-Hermitian matrices. We present the standard
terminology and some basic results on Krein spaces.
The reader is referred to \cite{Ando, Azizov, Bogner} for a complete exposition on the subject.\\

Let $(\mathscr{H},\langle\cdot,\cdot\rangle)$ be a Hilbert space
and $\mathbb{B}(\mathscr{H})$ denote the $C^*$-algebra of all
bounded linear operators on $\mathscr{H}$ with the identity
operator $I_{\mathscr{H}}$. An operator
$T\in\mathbb{B}(\mathscr{H})$ is called positive if $\langle
Tx,x\rangle\geq 0$  for all $x\in\mathscr{H}$. If $T$ is a
positive invertible operator we write $T>0$.
For bounded selfadjoint operators $T$ and $S$ on $\mathscr{H}$, we say $T\leq S$ if $S-T\geq 0$.\\
Suppose that a nontrivial selfadjoint involution $J$ on
$\mathscr{H}$, i.e. $J=J^*=J^{-1}$, is given to produce an
indefinite inner product
\begin{eqnarray*}
[x,y]_J:=\langle Jx,y\rangle\qquad(x,y\in \mathscr{H}).
\end{eqnarray*}
In this case $(\mathscr{H},J)$ is called a Krein space. The
operators ${\rm P}_{+}=\frac{I+J}{2}$ and ${\rm
P}_{-}=\frac{I-J}{2}$ are orthogonal projections onto
$\mathscr{H}_{+}={\rm ran}({\rm P}_{+})$ and $\mathscr{H}_{-}={\rm
ran}({\rm P}_{-})$, respectively and
\begin{eqnarray*}
\mathscr{H}=\mathscr{H}_{+}\oplus \mathscr{H}_{-}.
\end{eqnarray*}
In correspondence to this orthogonal decomposition, each bounded
linear operator $C$ on $\mathscr{H}$ is uniquely represented by
the matrix
\begin{eqnarray}
C=\left( \begin{array}{cc} C_{11} & C_{12}\\
C_{21} & C_{22}\end{array}\right)
\end{eqnarray}
where $C_{11}={\rm P}_{+}C{\rm P}_{+}|_{\mathscr{H}_+}\,,\,C_{12}={\rm P}_{+}C{\rm P}_{-}|_{\mathscr{H}_{-}}
\,,\,C_{21}={\rm P}_{-}C{\rm P}_{+}|_{\mathscr{H}_{+}}\,,\,C_{22}={\rm P}_{-}C{\rm P}_{-}|_{\mathscr{H}_{-}}$.
The $n$-dimensional Minkowski space is a well-known example of a Krein space:
\begin{example}
Let $M_n({\mathbb C})$ be the set of all complex $n\times n$
matrices and let $\langle\cdot,\cdot\rangle$ be the standard inner
product on ${\mathbb C}^n$. For selfadjoint involution
\begin{eqnarray*}
J_0=\left( \begin{array}{cc} I_{n-1} & 0\\
0 & -1\end{array}\right),
\end{eqnarray*}
where $I_{n-1}$ denotes the identity of $M_{n-1}(\mathbb{C})$, one
can define an indefinite inner product $[.,.]_{J_0}$ on ${\mathbb
C}^n$ by
\begin{eqnarray*}
[x,y]_{J_0}=\langle J_0x,y\rangle=\sum_{k=1}^{n-1}x_k\bar{y}_k-x_n\bar{y}_n
\end{eqnarray*}
for $x=(x_1,\cdots,x_n)$ , $y=(y_1,\cdots,y_n)\in {\mathbb C}^n$.
The Krein space $({\mathbb C}^n, J_0)$ is called the
$n$-dimensional Minkowski space.
\end{example}
Let $(\mathscr{H}_1,J_1)$ and $(\mathscr{H}_2,J_2)$ be Krein spaces. The $(J_1,J_2)$-adjoint operator
$A^{\sharp}$ of $A\in\mathbb{B}(\mathscr{H}_1,\mathscr{H}_2)$ is defined by
\begin{eqnarray*}
[Ax,y]_{J_2}=[x,A^{\sharp}y]_{J_1}\qquad(x\in\mathscr{H}_1,y\in\mathscr{H}_2),
\end{eqnarray*}
which is equivalent to say that $A^{\sharp}=J_1A^*J_2$. Trivially
$(A^{\sharp})^{\sharp}=A$. An operator
$A\in\mathbb{B}(\mathscr{H})$ on a Krein space $(\mathscr{H},J)$
is said to be $J$-selfadjoint if $A^{\sharp}=A$, or equivalently,
$A=JA^*J$.
The spectrum of a $J$-selfadjoint operator on a Krein space $(\mathscr{H},J)$
is not necessarily real (it can even cover the whole plane); see \cite{Bogner}.\\
For $J$-selfadjoint operators $A$ and $B$, the $J$-order, denoted as $A\leq^{J}B$, is defined by
\begin{eqnarray*}
[Ax,x]_J\leq[Bx,x]_J\qquad(x\in \mathscr{H}).
\end{eqnarray*}
Clearly $A\leq^{J}B$ if and only if $JA\leq JB$. The
$J$-selfadjoint operator $A\in\mathbb{B}(\mathscr{H})$ is said to
be $J$-positive if $A\geq^{J}0$. Evidently, $A$ is $J$-positive if
and only if $AJ$ is positive.
It is easy to see that neither $A\geq0$ implies $A\geq^{J}0$ nor $A\geq^{J}0$ implies $A\geq0$.\\
Let $(\mathscr{H}_1,J_1)$ and $(\mathscr{H}_2,J_2)$ be Krein
spaces. As usual, let $\mathbb{B}(\mathscr{H}_1,\mathscr{H}_2)$ be
the space of all bounded linear operators from $\mathscr{H}_1$
into $\mathscr{H}_2$.  An operator
$C\in\mathbb{B}(\mathscr{H}_1,\mathscr{H}_2)$ is called a
$(J_1,J_2)$-contraction if
$C^{\sharp}C\leq^{J_1}I_{\mathscr{H}_1}$, that is,
$C^{*}J_{2}C\leq J_1$. The operator $C$ is called
$(J_1,J_2)$-bicontraction if $C$ and $C^{\sharp}$ are
$(J_1,J_2)$-contraction and $(J_2,J_1)$-contraction, respectively.
In the case that $\mathscr{H}_1=\mathscr{H}_2$ and $J_1=J_2=J$ we
write $J$-contraction instead of $(J,J)$-contraction. An
invertible operator $U\in\mathbb{B}(\mathscr{H}_1,\mathscr{H}_2)$
such that
$U^{\sharp}=U^{-1}$ is said to be $(J_1,J_2)$-unitary.\\

Note that in contrast to the setting of Hilbert spaces, not all
$J$-contractions are $J$-bicontractions; see \cite[Example
1.3.8]{Rov2}. The following theorem presenting a suitable
condition for a $J$-contraction to being a $J$-bicontraction.
\begin{theorem}\cite[Corollary 3.3.3]{Ando}
A $J$-contraction $C$ on a Krein space $(\mathscr{H},J)$ is a
$J$-bicontraction if and only if the operator $C_{11}$ in the matrix
form (1.1) of $C$ is invertible. In particular if $C$ is an
invertible $J$-contraction, then $C$ is a $J$-bicontraction.
\end{theorem}
Let $C\in\mathbb{B}(\mathscr{H}_1,\mathscr{H}_2)$. By a defect
operator for $C$ we mean any operator
$E\in\mathbb{B}(\tilde{\mathscr{H}_2},\mathscr{H}_1)$, where
$(\tilde{\mathscr{H}_2},\tilde{J_2})$ is a Krein space, such that
$E$ has zero kernel and
$I_{\mathscr{H}_1}-C^{\sharp}C=EE^{\sharp}$. A Julia operator for
$C$ is a $(J_1\oplus\tilde{J_1}, J_2\oplus\tilde{J_2})$-unitary
$U:\mathscr{H}_1\oplus\tilde{\mathscr{H}_1}\rightarrow\mathscr{H}_2\oplus\tilde{\mathscr{H}_2}$
of the form
\begin{eqnarray*}
U=\left( \begin{array}{cc} C & D\\
E^{\sharp} & -L^{\sharp}\end{array}\right),
\end{eqnarray*}
where $(\tilde{\mathscr{H}_1}, \tilde{J_1})$ and
$(\tilde{\mathscr{H}_2},\tilde{J_2})$ are Krein spaces such that
the operators
$D\in\mathbb{B}(\tilde{\mathscr{H}_1},\mathscr{H}_2)$ and
$E\in\mathbb{B}(\tilde{\mathscr{H}_2},\mathscr{H}_1)$
have zero kernels. In this case $E$ is a defect operator for $C$,
and $D$ is a defect operator for $C^{\sharp}$; cf. \cite{Rov1}.\\

Defect and Julia operators have played important roles in the
Krein space operator theory. The first constructions of these
operators in the Krein space setting are due to Arsene et al
\cite{Arsene}. An abstract theory of Julia operators in Krein
spaces and its applications appear in a number of sources; see
e.g. \cite{Rov1,Rov2,Rov3}. We need the following theorem as a
result of Corollary 1.4.3, Theorem 2.4.5 of \cite{Rov2} and
Theorem 13 of \cite{Rov3}.
\begin{theorem}
Suppose that $(\mathscr{H}_1,J_1)$ and $(\mathscr{H}_2,J_2)$ are
Krein spaces and $C\in\mathbb{B}(\mathscr{H}_1,\mathscr{H}_2)$ is
an injective $(J_1,J_2)$-bicontraction. Then $C$ has a unique (up to unitary)
Julia operator of the form
\begin{eqnarray*}
U=\left( \begin{array}{cc} C & D\\
E^{\sharp} &
-L^{*}\end{array}\right)\in\mathbb{B}(\mathscr{H}_1\oplus\tilde{\mathscr{H}_1},\mathscr{H}_2\oplus\tilde{\mathscr{H}_2}),
\end{eqnarray*}
for which $\tilde{\mathscr{H}_1}$ and $\tilde{\mathscr{H}_2}$ are
Hilbert spaces and $-L^{*}$ is a Hilbert space contraction.
\end{theorem}
A real valued continuous function $f$ on an interval $\mathcal{I}\subseteq\mathbb{R}$ is said to be operator convex if
\begin{eqnarray*}
f\Big((1-\lambda)A+\lambda B\Big)\leq (1-\lambda)f(A)+\lambda
f(B),
\end{eqnarray*}
for all $\lambda\in [0,1]$ and all selfadjoint operators $A$ and
$B$ on a Hilbert space $\mathscr{H}$, whose spectra are contained
in $\mathcal{I}$. The notion of operator/matrix convex functions, as well as that of operator/matrix
monotone functions, have played a vital role in operator/matrix analysis and its applications,
e.g., to quantum information. On the other hand, the operator theory in spaces
with an indefinite inner product (notably in Pontryagin spaces and Krein spaces) has
been investigated for a long time with the aim to establish mathematical formalism of
quantum field theory. An indefinite analogue of the concept of monotone matrix function was studied in \cite{ALP}.\\
In this paper we introduce the notion of Krein-operator convexity in the setting
of Krein spaces. We also present an
indefinite version of Jensen's operator inequality based on the
ideas due to Hansen and Pedersen \cite{Hansen}.


\section{Main results}
Let $A$ be a $J$-selfadjoint operator on a Krein space
$(\mathscr{H},J)$ and let $\mathcal{U}$ be an open set which is
not necessarily connected such that
$\sigma(A)\subseteq\mathcal{U}$. Suppose that
$f:\mathcal{U}\rightarrow\mathbb{C}$ is an analytic function. Then
the operator $f(A)$ is defined by the usual Dunford-Riesz integral
\begin{eqnarray}
f(A)=\frac{1}{2\pi i}\int_{\Gamma}f(\lambda)(\lambda I_{\mathscr{H}}-A)^{-1}d\lambda,
\end{eqnarray}
where $\Gamma$ is a suitable finite family of closed rectifiable contours with positive direction surrounding
$\sigma(A)$ in its interior; see \cite{Rudin}\\
From now on assume that $\mathcal{U}$ is an open set (not
necessarily connected) in the plane, which is symmetric with
respect to real axis and $\mathcal{U}\cap\mathbb{R}$ consists of a
segment of real axis. The following proposition provide some
conditions for $f(A)$ to be $J$-selfadjoint whenever $A$ is
$J$-selfadjoint.
\begin{proposition}
Let $A$ be a $J$-selfadjoint operator on a Krein space
$(\mathscr{H},J)$ such that $\sigma(A)\subseteq\mathcal{U}$. If
$f:\mathcal{U}\rightarrow\mathbb{C}$ is an analytic function such
that $f(x)$ is real for all $x\in\mathcal{U}\cap\mathbb{R}$, then
$f(A)$ is $J$-selfadjoint.
\end{proposition}
\begin{proof}
By the definition of a $J$-selfadjoint operator, $AJ=JA^*$. Hence
\begin{eqnarray}
Jf(A)&=&\frac{1}{2\pi i}\int_{\Gamma}f(\lambda)J(\lambda I_{\mathscr{H}}-A)^{-1}d\lambda\nonumber\\
&=&\frac{1}{2\pi i}\int_{\Gamma}f(\lambda)(\lambda J-AJ)^{-1}d\lambda\nonumber\\
&=&\frac{1}{2\pi i}\int_{\Gamma}f(\lambda)\Big(J(\lambda I_{\mathscr{H}}-A^*)\Big)^{-1}d\lambda\nonumber\\
&=&\frac{1}{2\pi i}\int_{\Gamma}f(\lambda)(\lambda I_{\mathscr{H}}-A^*)^{-1}Jd\lambda\nonumber\\
&=&f(A^*)J.
\end{eqnarray}
The spectrum of a $J$-selfadjoint operator on a Krein space
$(\mathscr{H},J)$ is symmetric with respect to the real axis (see
\cite[Corollary 6.3]{Bogner}). Suppose that $\Gamma$ is a closed
rectifiable contour with positive direction surrounding
$\sigma(A)$ in its interior such that it is symmetric with respect
to the real axis. By the change of variable $\lambda$ to
$\bar{\lambda}$, we have
\begin{eqnarray*}
f(A)=\frac{1}{2\pi
i}\int_{\Gamma^{-1}}f(\bar{\lambda})(\bar{\lambda}I_{\mathscr{H}}-A)^{-1}d\bar{\lambda}=-\frac{1}{2\pi
i}\int_{\Gamma}f(\bar{\lambda})(\bar{\lambda}I_{\mathscr{H}}-A)^{-1}d\bar{\lambda}.
\end{eqnarray*}
By the reflection principle, we have $\overline{f(\bar{z})}=f(z)$ for all $z\in\Gamma$. Therefore
\begin{eqnarray*}
(f(A))^*&=&\Big(-\frac{1}{2\pi i}\int_{\Gamma}f(\bar{\lambda})(\bar{\lambda}I_{\mathscr{H}}-A)^{-1}d\bar{\lambda}\Big)^*\nonumber\\
&=&\frac{1}{2\pi i}\int_{\Gamma}\overline{f(\bar{\lambda})}(\lambda I_{\mathscr{H}}-A^*)^{-1}d\lambda\nonumber\\
&=&\frac{1}{2\pi i}\int_{\Gamma}f(\lambda)(\lambda I_{\mathscr{H}}-A^*)^{-1}d\lambda\nonumber\\
&=&f(A^*).
\end{eqnarray*}
It follows from (2.2) that $Jf(A)=f(A^*)J=f(A)^*J$. Hence $f(A)$ is $J$-selfadjoint.\\
\end{proof}
The notion of operator convexity is a generalization of that of
usual convexity. This is based on the fact that the selfadjoint
operators (Hermitian matrices) can be regarded as a generalization
of the real numbers. We aim to introduce the notion of Krein-operator convexity
as a generalization of the operator convexity.\\
One may immediately say that a real valued function $f$ being
analytic on an interval $\mathcal{I}$ is Krein-operator convex if
\begin{eqnarray*}
f\Big((1-\lambda)A+\lambda B\Big)\leq^{J} (1-\lambda)f(A)+\lambda f(B),
\end{eqnarray*}
for all $\lambda\in [0,1]$ and all $J$-selfadjoint operators $A$ and $B$ on any Krein space $(\mathscr{H},J)$ with spectra contained in $\mathcal{I}$.
As noticed by Ando \cite{Ando1}, this definition is vain. In fact if $\mathscr{H}=\mathscr{H}_+\oplus\mathscr{H}_-$ and
$J=\left( \begin{array}{cc} I_{\mathscr{H}_+}& 0\\
0 & -I_{\mathscr{H}_-}\end{array}\right)$ with respect to this
decomposition, $\alpha,\beta\in\mathcal{I}$ and consider $A=\alpha
I_{\mathscr{H}}$ and $B=\beta I_{\mathscr{H}}$, then $A,B$ are $J$-selfadjoint (selfadjoint)
and $\sigma(A)=\{\alpha\}\subseteq\mathcal{I}$ and
$\sigma(B)=\{\beta\}\subseteq\mathcal{I}$. Hence
\begin{eqnarray*}
f\Big((1-\lambda)A+\lambda B\Big)=f\Big((1-\lambda)\alpha+\lambda\beta\Big)I_{\mathscr{H}}.
\end{eqnarray*}
and
\begin{eqnarray*}
(1-\lambda)f(A)+\lambda f(B)=\Big((1-\lambda)f(\alpha)+\lambda f(\beta)\Big)I_{\mathscr{H}},
\end{eqnarray*}
for all $\lambda\in[0,1]$. Therefore $f\Big((1-\lambda)A+\lambda B\Big)\leq^{J} (1-\lambda)f(A)+\lambda f(B)$ whence
\begin{eqnarray*}
\Big((1-\lambda)f(\alpha)+\lambda
f(\beta)\Big)-f\Big((1-\lambda)\alpha+\lambda\beta\big)\Big)\geq 0
\end{eqnarray*}
and
\begin{eqnarray*}
\Big((1-\lambda)f(\alpha)+\lambda f(\beta)\Big)-f\Big((1-\lambda)\alpha+\lambda\beta\big)\Big)\leq 0.
\end{eqnarray*}
Therefore
\begin{eqnarray*}
(1-\lambda)f(\alpha)+\lambda f(\beta)=f\Big((1-\lambda)\alpha+\lambda\beta\Big),
\end{eqnarray*}
for all $\alpha,\beta\in\mathcal{I}$ an all $\lambda\in[0,1]$.
Thus $f$ is linear on $\mathcal{I}$, that is, $f(t)=at+b$ for some $a,b\in\mathbb{R}$.\\
To avoid such trivialities, we restrict ourselves to the $J$-positive operators instead of $J$-selfadjoint operators.
It is well known that the spectrum of a $J$-positive operator on a
Krein space $(\mathscr{H},J)$ is real (see \cite[Theorem
2.1]{Ando}). From this point of view, the $J$-positive operators
seem to behave like the selfadjoint operators on Hilbert spaces.
The following definition seems to be satisfactory.
\begin{definition}
Suppose that $f:\mathcal{U}\rightarrow\mathbb{C}$ is an analytic function
such that $f(x)$ is real for all $x\in\mathcal{U}\cap\mathbb{R}$.
Then $f$ is said to be Krein-operator convex if
\begin{eqnarray}
f\Big((1-\lambda)A+\lambda B\Big)\leq^{J} (1-\lambda)f(A)+\lambda
f(B),
\end{eqnarray}
for all $\lambda\in [0,1]$ and all $J$-positive operators $A$ and
$B$ on any Krein space $(\mathscr{H},J)$, such that spectra of
$A$, $B$ and
$(1-\lambda)A+\lambda B$ are contained in $\mathcal{U}$.\\
Also, the condition (2.3) can be obviously replaced by
\begin{eqnarray*}
f\left(\frac{A+B}{2}\right)\leq^{J}\frac{f(A)+f(B)}{2}.
\end{eqnarray*}
\end{definition}
The following example shows that operator convex functions are not necessarily Krein-operator convex.
\begin{example}
Consider the $2$-dimensional Minkowski space $(\mathbb{C}^2,J_0)$ with $J_0=\left( \begin{array}{cc}1 & 0\\
0 & -1\end{array}\right)$. Let $A=\left( \begin{array}{cc} a_{11} & a_{12}\\
a_{21} & a_{22}\end{array}\right)\in M_2(\mathbb{C})$. Since the
$J_0$-selfadjoitness of $A$ is equivalent to the usual
selfadjointness of $J_0A$,
we have $A=\left( \begin{array}{cc} a_{11} & a_{12}\\
-\overline{a_{12}} & a_{22}\end{array}\right)$ in which $a_{11}$ and $a_{22}$ are real.
Let $A=\left( \begin{array}{cc} 1 & -1\\
1 & -2\end{array}\right)$ and $B=\left( \begin{array}{cc}1 & -1\\
1 & -3\end{array}\right)$. It is easy to see that $A$ and $B$ are
$J_0$-positive and they have real eigenvalues. Also we have
\begin{eqnarray*}
J_0\left(\frac{A^2+B^2}{2}\right)-J_0\left(\frac{A+B}{2}\right)^2=\left( \begin{array}{cc}0 & 0\\
0 & -\frac{1}{4}\end{array}\right)\ngeq0.
\end{eqnarray*}
It shows that
\begin{eqnarray*}
\left(\frac{A+B}{2}\right)^2\nleq^{J_0}\frac{A^2+B^2}{2}.
\end{eqnarray*}
Therefore $f(t)=t^2$ is not Krein-operator convex, while it is
well known that, this function is operator convex on $\mathbb{R}$
(see \cite[p. 8]{Mond}).
\end{example}
The following lemma characterize the invertible $J$-positive
operators on a Krein space $(\mathscr{H},J)$. For more information
of $J$-positivite operators see \cite[Chapter 2]{Ando}.
\begin{proposition}
Let $(\mathscr{H},J)$ be a Krein space and
$A\in\mathbb{B}(\mathscr{H})$. Then $A$ is $J$-positive and
invertible if and only if it is of the form $A=J\tilde{A}$ for
some $\tilde{A}>0$.
\end{proposition}
\begin{proof}
Let $A=J\tilde{A}$ with $\tilde{A}>0$. Then $A$ is invertible and $JA=J(J\tilde{A})=\tilde{A}\geq 0$.
It follows that $A$ is $J$-positive and invertible.\\
Conversely, if $A$ is $J$-positive and invertible, then
$A=J\tilde{A}$ with $\tilde{A}:=JA$ which is positive and
invertible.
\end{proof}
\begin{example} \cite{Ando1}
Define the function $f(z)=\frac{1}{z}$ ($z\neq0$) and $f(0)=0$.
Then $f$ is Krein-operator convex on
$\mathcal{U}=\mathbb{C}\setminus\{0\}$. To see this, assume that
$A$ and $B$ are invertible $J$-positive operators such that the
spectra of $A$, $B$ and $(1-\lambda)A+\lambda B$ are contained in
$\mathcal{U}\cap\mathbb{R}=\mathbb{R}\setminus\{0\}$. Using the
Dunford-Riesz representation (2.1), $\Big((1-\lambda)A+\lambda
B)\Big)^{-1}$ exists. By Proposition 2.4, $A=J\tilde{A}$ and
$B=J\tilde{B}$ for some $\tilde{A}>0$ and $\tilde{B}>0$. The
restriction of $f$ to $(0,\infty)$ is operator convex \cite[p.
8]{Mond}. Therefore
\begin{eqnarray*}
\Big((1-\lambda)A+\lambda B\Big)^{-1}J&=&\left((1-\lambda)\tilde{A}+\lambda \tilde{B}\right)^{-1}\nonumber\\
&\leq&(1-\lambda)\tilde{A}^{-1}+\lambda\tilde{B}^{-1}\nonumber\\
&=&\Big((1-\lambda)A^{-1}+\lambda B^{-1}\Big)J.\nonumber
\end{eqnarray*}
Hence
\begin{eqnarray*}
\Big((1-\lambda)A+\lambda B\Big)^{-1}\leq^{J}(1-\lambda)A^{-1}+\lambda B^{-1}.
\end{eqnarray*}
\end{example}

Generally, suppose that $\mathcal{U}$ contains an interval
$\mathcal{I}\subseteq[0,\infty)$ and $f$ is a Krein-operator
convex
function on $\mathcal{U}$. Let $\tilde{A}=\left( \begin{array}{cc} A & 0\\
0& 0\end{array}\right)$ and $\tilde{B}=\left( \begin{array}{cc} B & 0\\
0& 0\end{array}\right)$, for which $A$ and $B$ are positive
operators on a Hilbert space $\mathscr{H}$. Then
\begin{eqnarray*}
f\Big((1-\lambda)\tilde{A}+\lambda\tilde{B}\Big)\leq^{J}(1-\lambda)f(\tilde{A})+\lambda
f(\tilde{B})
\end{eqnarray*}
with $J=\left( \begin{array}{cc} I_{\mathscr{H}} & 0\\
0& -I_{\mathscr{H}}\end{array}\right)$ for all $\lambda\in[0,1]$.
It follows that $f\big((1-\lambda)A+\lambda
B\big)\leq(1-\lambda)f(A)+\lambda f(B)$ for all positive operators
$A$ and $B$ with spectra in $\mathcal{I}$ and all $\lambda\in[0,1]$. Hence the restriction of
$f$ to $\mathcal{I}$ is operator convex. Thus under mild condition,
when the Krein space happens to be a Hilbert space, the definition
of a Krein-operator convex function then coincids with the
traditional definition, when we restricted
ourself to the class of positive operators.\\

Before we present our main result, we give a lemma needed later.
\begin{lemma}
Let $(\mathscr{H},J)$ be a Krein space and $A$ be a $J$-selfadjoint operator with $\sigma(A)\subseteq\mathcal{U}$. Then
\begin{eqnarray*}
f(U^{\sharp}AU)=U^{\sharp}f(A)U
\end{eqnarray*}
for any analytic function $f:\mathcal{U}\rightarrow\mathbb{C}$ and any $J$-unitary $U$.
\end{lemma}
\begin{proof}
The operator $U^{\sharp}AU$ is $J$-selfadjoint, since
\begin{eqnarray*}
(U^{\sharp}AU)^{\sharp}&=&J(U^{\sharp}AU)^*J=J(U^*A^*JUJ)J=JU^*A^*JU=JU^*JAU=U^{\sharp}AU,
\end{eqnarray*}
where we use the $J$-selfadjointness of $A$ at the fourth
equality. By the elementary operator theory, if $A$ is some
invertible operator, then
$\sigma(U^{\sharp}AU)\cup\{0\}=\sigma(AUU^{\sharp})\cup\{0\}=\sigma(A)\cup\{0\}$,
and since $U^{\sharp}AU$ is invertible if $A$ is, it follows that
these operators have the same spectra in this case. Finally
\begin{eqnarray}
f(U^{\sharp}AU)&=&\frac{1}{2\pi i}\int_{\Gamma}f(\lambda)(\lambda I_{\mathscr{H}}-U^{\sharp}AU)^{-1}d\lambda\nonumber\\
&=&\frac{1}{2\pi i}\int_{\Gamma}f(\lambda)U^{\sharp}(\lambda I_{\mathscr{H}}-A)^{-1}Ud\lambda\nonumber\\
&=&U^{\sharp}f(A)U.\nonumber
\end{eqnarray}
\end{proof}

In the following theorem we present a Jensen type inequality for
Krein-operator convex functions in the setting of Krein spaces.
\begin{theorem}
Let $f$ be a Krein-operator convex function on $\mathcal{U}$ and
$f(0)=0$. Then
\begin{eqnarray}
f(C^{\sharp}AC)\leq^{J}C^{\sharp}f(A)C
\end{eqnarray}
for all $J$-positive operators $A$  and all invertible
$J$-contractions $C$ on a Krein space $(\mathscr{H},J)$ such that
the spectra of $A$, $C^{\sharp}AC$ and $D^{\sharp}AD$ are
contained in $\mathcal{U}$, where $D$ is a defect operator for
$C^{\sharp}$.
\end{theorem}
\begin{proof}
Let $f:\mathcal{U}\rightarrow\mathbb{C}$ be a Krein-operator
convex function and let $A$ be a $J$-positive operator with
$\sigma(A)\subseteq\mathcal{U}$. Also assume that $C$ is an
invertible $J$-contraction. By Theorem 1.2, $C$ is a
$J$-bicontraction. Therefore Theorem 1.3 implies that $C$ has a
unique Julia operator
\begin{eqnarray*}
U=\left( \begin{array}{cc} C & D\\
E^{\sharp} & -L^*\end{array}\right)\in\mathbb{B}(\mathscr{H}\oplus\tilde{\mathscr{H}_1},\mathscr{H}\oplus\tilde{\mathscr{H}_2}),
\end{eqnarray*}
where $\tilde{\mathscr{H}_1}$ and $\tilde{\mathscr{H}_2}$ are Hilbert spaces
such that the operators $D\in\mathbb{B}(\tilde{\mathscr{H}_1},\mathscr{H})$
and $E\in\mathbb{B}(\tilde{\mathscr{H}_2},\mathscr{H})$ have zero kernels
and $-L^{*}$ is a Hilbert space contraction. It is well known that
$\mathscr{H}\oplus\tilde{\mathscr{H}_i}$ are Krein spaces with fundamental symmetries
${\bf\tilde{J}_i}=J\oplus I_{\tilde{\mathscr{H}_i}}=\left( \begin{array}{cc} J & 0\\
0 & I_{\tilde{\mathscr{H}_i}}\end{array}\right)$ for $i=1,2$. Then
$U$ is a $({\bf\tilde{J}_1},{\bf\tilde{J}_2})$-unitary i.e.
$U^{\sharp}U=UU^{\sharp}=I_{\mathscr{H}\oplus\tilde{\mathscr{H}_1}}$.
Let
\begin{eqnarray*}
V=\left( \begin{array}{cc} C & -D\\
E^{\sharp} & L^*\end{array}\right).
\end{eqnarray*}
Note that $V$ is simply $U$ multiplied by fundamental symmetry
$I_{\mathscr{H}}\oplus -I_{\tilde{\mathscr{H}_1}}$. So $V$ is a
$({\bf\tilde{J}_1},{\bf\tilde{J}_2})$-unitary. Let
\begin{eqnarray*}
X=\left( \begin{array}{cc} A & 0\\
0 & 0\end{array}\right):\mathscr{H}\oplus\tilde{\mathscr{H}_2}\rightarrow \mathscr{H}\oplus\tilde{\mathscr{H}_2}.
\end{eqnarray*}
Since
\begin{eqnarray*}
X^{\sharp}={\bf\tilde{J}_2}X^*{\bf\tilde{J}_2}=\left( \begin{array}{cc} J & 0\\
0 & I_{\tilde{\mathscr{H}_2}}\end{array}\right)\left( \begin{array}{cc} A^* & 0\\
0 & 0\end{array}\right)\left( \begin{array}{cc} J & 0\\
0 & I_{\tilde{\mathscr{H}_2}}\end{array}\right)=\left( \begin{array}{cc} A^{\sharp} & 0\\
0 & 0\end{array}\right)=\left( \begin{array}{cc} A & 0\\
0 & 0\end{array}\right)=X,
\end{eqnarray*}
we conclude that $X$ is ${\bf\tilde{J}_2}$-selfadjoint. Moreover,
\begin{eqnarray*}
{\bf\tilde{J}_2}X=\left( \begin{array}{cc} J & 0\\
0 & I_{\tilde{\mathscr{H}_2}}\end{array}\right)\left( \begin{array}{cc} A & 0\\
0 & 0\end{array}\right)=\left( \begin{array}{cc} JA & 0\\
0 & 0\end{array}\right)\geq 0.
\end{eqnarray*}
Hence $X$ is ${\bf\tilde{J}_2}$-positive. It is clear that
$U^{\sharp}XU$ and $V^{\sharp}XV$ are ${\bf\tilde{J}_1}$-positive
operators. If $\lambda\notin\sigma(A)\cup\{0\}$, then we have
\begin{eqnarray*}
(\lambda I_{\mathscr{H}\oplus\tilde{\mathscr{H}_2}}-X)^{-1}
=\left( \begin{array}{cc}(\lambda I_{\mathscr{H}}-A)^{-1} & 0\\
0 & \lambda^{-1}I_{\tilde{\mathscr{H}_2}}\end{array}\right),
\end{eqnarray*}
so that $\lambda\notin\sigma(X)$. Hence
$\sigma(X)\subseteq\mathcal{U}$. Since
$\sigma(U^{\sharp}XU)=\sigma(V^{\sharp}XV)=\sigma(X)$,
the spectra of ${\bf\tilde{J}_1}$-positive operators
$U^{\sharp}XU$ and $V^{\sharp}XV$ are contained in $\mathcal{U}$.
Consequently, by the Krein-operator convexity of $f$ and Lemma
2.6, we infer that
\begin{eqnarray}
\left( \begin{array}{cc} f(C^{\sharp}AC) & 0\\
0& f(D^{\sharp}AD)\end{array}\right)&=&f\left( \begin{array}{cc} C^{\sharp}AC & 0\\
0& D^{\sharp}AD\end{array}\right)\nonumber\\
&=&f\left(\frac{U^{\sharp}XU+V^{\sharp}XV}{2}\right)\nonumber\\
&\leq^{{\bf\tilde{J}_1}}&\frac{f(U^{\sharp}XU)+f(V^{\sharp}XV)}{2}\nonumber\\
&=&\frac{U^{\sharp}f(X)U+V^{\sharp}f(X)V}{2}\nonumber\\
&=&\frac{1}{2}U^{\sharp}\left( \begin{array}{cc} f(A)& 0\\
0& f(0)\end{array}\right)U+\frac{1}{2}V^{\sharp}\left( \begin{array}{cc} f(A)& 0\\
0& f(0)\end{array}\right)V\nonumber\\
&=&\frac{1}{2}U^{\sharp}\left( \begin{array}{cc} f(A)& 0\\
0& 0\end{array}\right)U+\frac{1}{2}V^{\sharp}\left( \begin{array}{cc} f(A)& 0\\
0& 0\end{array}\right)V\nonumber\\
&=&\left( \begin{array}{cc} C^{\sharp}f(A)C & 0\\
0& D^{\sharp}f(A)D\end{array}\right).\nonumber
\end{eqnarray}
Hence
\begin{eqnarray*}
\left( \begin{array}{cc} J & 0\\
0& I_{\tilde{\mathscr{H}_1}}\end{array}\right)\left( \begin{array}{cc} f(C^{\sharp}AC) & 0\\
0& f(D^{\sharp}AD)\end{array}\right)\leq\left( \begin{array}{cc} J & 0\\
0& I_{\tilde{\mathscr{H}_1}}\end{array}\right)\left( \begin{array}{cc} C^{\sharp}f(A)C & 0\\
0& D^{\sharp}f(A)D\end{array}\right).
\end{eqnarray*}
It follows that $Jf(C^{\sharp}AC)\leq JC^{\sharp}f(A)C$. Therefore $f(C^{\sharp}AC)\leq^{J}C^{\sharp}f(A)C$.
\end{proof}
\begin{remark}
In the classical case the usual operator convexity is equivalent to
\begin{eqnarray*}
f(C^*AC)\leq C^*f(A)C,
\end{eqnarray*}
where $A$ is selfadjoint and $C$ is an isometry; see \cite[Theorem
1.9]{Mond} or \cite[Section V]{BHA}. Validity of an analogous
relation as $f(C^{\sharp}AC)\leq^{J}C^{\sharp}f(A)C$ for all
$J$-selfadjoint $A$ and all invertible $J$-isometries $C$ (i.e.
$C^{\sharp}C=I$ ) on a Krein space $(\mathscr{H},J)$ is vain,
since if $C^{\sharp}C=I$, then $CC^{\sharp}=I$ and so
$C^{\sharp}=C^{-1}$ and we enter into a trivial situation.
\end{remark}
In the following example we show that in contrast with usual
operator convex functions, the condition (2.4) in Theorem 2.7 is
not equivalent to the Krein-operator convexity of $f$.
\begin{example}
Let $A$ be a $J$-positive operator and $C$ be an invertible $J$-contraction
on a Krein space $(\mathscr{H},J)$. Since $A$ is $J$-selfadjoint
and $CC^{\sharp}\leq^{J}I_{\mathscr{H}}$, we have
\begin{eqnarray*}
(C^{\sharp}AC)^2=(C^{\sharp}A^{\sharp})CC^{\sharp}(AC)=
(AC)^{\sharp}CC^{\sharp}AC\leq^{J}(AC)^{\sharp}I_{\mathscr{H}}AC=C^{\sharp}AAC.
\end{eqnarray*}
It follows that $(C^{\sharp}AC)^2\leq^{J}C^{\sharp}A^2C$.\\
Therefore the function $f(t)=t^2$ satisfies (2.4). By Example 2.3, this function, however, is not Krein-operator convex.
\end{example}

\bibliographystyle{amsplain}

\end{document}